\documentclass[11pt,english]{article}

\usepackage[margin = 2.5cm]{geometry}

\usepackage{amsthm}
\usepackage{amsmath}
\usepackage{amssymb}
\usepackage{setspace}
\usepackage{mathtools}
\usepackage{graphicx}
\usepackage[hidelinks]{hyperref}
\usepackage{thm-restate}
\usepackage{cleveref}
\usepackage{enumitem}
\setlist[itemize]{leftmargin=*} 
\usepackage{framed}

\usepackage{caption}
\usepackage{floatrow}
\usepackage{algorithmic,algorithm}

\usepackage[square,sort,comma,numbers]{natbib} 
\theoremstyle{plain}

\newtheorem*{thm*}{Theorem}
\newtheorem{thm}{Theorem}
\Crefname{thm}{Theorem}{Theorems}

\newtheorem*{lem*}{Lemma}

\Crefname{lem}{Lemma}{Lemmas}

\newtheorem*{claim*}{Claim}
\newtheorem{claim}[thm]{Claim}
\crefname{claim}{Claim}{Claims}
\Crefname{claim}{Claim}{Claims}

\Crefname{prop}{Proposition}{Propositions}

\crefname{cor}{Corollary}{Corollaries}

\crefname{conj}{Conjecture}{Conjectures}

\Crefname{qn}{Question}{Questions}

\Crefname{obs}{Observation}{Observations}

\Crefname{ex}{Example}{Examples}

\theoremstyle{definition}

\Crefname{prob}{Problem}{Problems}

\newtheorem{defn}[thm]{Definition}
\Crefname{defn}{Definition}{Definitions}

\theoremstyle{remark}

\renewenvironment{proof}[1][]{\begin{trivlist}
\item[\hspace{\labelsep}{\bf\noindent Proof#1.\/}] }{\qed\end{trivlist}}

\newcommand{\remove}[1]{}

\newcommand{\floor}[1]{
    \lfloor #1 \rfloor
}

\DeclareMathOperator{\Ex}{ex}
\newcommand{\HH}{\mathcal{H}}
\renewcommand{\P}{\mathcal{P}}
\newcommand{\E}{\mathbb{E}}
\newcommand{\F}{\mathcal{F}}

\begin{document}


\title{Many $H$-copies in graphs with a forbidden tree}

\author{
    Shoham Letzter\thanks{
        ETH Institute for Theoretical Studies,
        ETH,
        8092 Zurich;
        e-mail: \texttt{shoham.letzter}@\texttt{eth-its.ethz.ch}.
		Research supported by Dr.~Max R\"ossler, the Walter Haefner Foundation
		and the ETH Zurich Foundation.
    }
}

\date{}
\maketitle

\begin{abstract}

	\setlength{\parskip}{\medskipamount}
    \setlength{\parindent}{0pt}
    \noindent

	For graphs $H$ and $F$, let $\Ex(n, H, F)$ be the maximum possible number of copies of $H$ in an $F$-free graph on $n$ vertices.
	The study of this function, which generalises the well-studied Tur\'an numbers of graphs, was initiated recently by Alon and Shikhelman. 
	We show that if $F$ is a tree then $\Ex(n, H, F) = \Theta(n^r)$ for some integer $r = r(H, F)$, thus answering one of their questions.

\end{abstract}

	\section{Introduction} \label{sec:intro}

		Given graphs $H$ and $F$ with no isolated vertices and an integer $n$, let $\Ex(n, H, F)$ be the maximum possible number of copies of $H$ in an $F$-free graph on $n$ vertices.
		This function was introduced recently by Alon and Shikhelman \cite{alon-shikhelman}.
		In the special case where $H = K_2$, this is the maximum possible number of edges in an $F$-free graph on $n$ vertices, known as the \emph{Tur\'an number} of $F$, which is one of the main topics in extremal graph theory (see e.g.\ \cite{turan-survey} for a survey). 
		
		A few instances of $\Ex(n, H, F)$, with $H \neq K_2$, where studied prior to \cite{alon-shikhelman}. The first of these is due to Erd\H{o}s \cite{erdos-turan-cliques} who determined $\Ex(n, K_r, K_s)$ for all $r$ and $s$ (see also \cite{bollobas-turan-cliques}). 
		
		A different example that has received considerable attention recently is $\Ex(n, C_r, C_s)$ for various values of $r$ and $s$. In 2008 Bollob\'as and Gy\H{o}ri \cite{bollobas-gyori} showed that $\Ex(n, K_3, C_5) = \Theta(n^{3/2})$, and their upper bound has been improved several times \cite{alon-shikhelman,ergemlidze-turan-triangle-pentagon}.
		Gy\H{o}ri and Li \cite{gyori-li} obtained upper and lower bounds on $\Ex(n, K_3, C_{2k+1})$, that were subsequently improved by F\"uredi and \"Ozkahaya \cite{furedi-ozkahaya} and by Alon and Shikhelman \cite{alon-shikhelman}. Moreover, the number $\Ex(n, C_5, K_3)$ was calculated precisely \cite{grzesik,hatami-pentagon-triangle}. Very recently, Gishboliner and Shapira \cite{gishboliner-shapira} determined $\Ex(n, C_r, C_s)$, up to a constant factor, for all $r > 3$, and, additionally, they studied $\Ex(n, K_3, C_s)$ for even $r$. Some additional more precise estimates for $\Ex(n, C_r, C_s)$ are known (see \cite{grzesik-kielak,gerbner-turan-even-cycles}).

		Another notable example is $\Ex(n, H, T)$ when $T$ is a tree. Alon and Shikhelman \cite{alon-shikhelman} showed that if $H$ is also a tree then the following holds.
		\begin{equation} \label{eqn:turan-tree-tree}
			\Ex(n, H, T) = \Theta(n^r) \text{\,\, for some integer $r = (F, H)$}.
		\end{equation}
		See also \cite{gyori-turan-path-path} for the study of the special case where $T$ and $H$ are paths. Alon and Shikhelman asked if \eqref{eqn:turan-tree-tree} still holds if only $T$ is required to be a tree (and $H$ is an arbitrary graph).

		\begin{thm} \label{thm:main}
			Let $H$ be a graph and let $T$ be a tree. Then there exists an integer $r = r(H, T)$ such that $\Ex(n, H, T) = \Theta(n^r)$.
		\end{thm}

		We present the proof in the next section, and conclude the paper in \Cref{sec:conc} with some closing remarks.

	\section{The proof} \label{sec:proof}

		Our aim is to prove that $\Ex(n, H, T) = \Theta(n^r)$ for some integer $r$. In fact, we shall prove this statement with an explicit value of $r$. For that we need the following two definitions.

		\begin{defn} \label{defn:blow-up}
			Given a graph $H$, a subset $U \subseteq V(H)$ and an integer $t$, the \emph{$(U, t)$-blow-up of $H$} is the graph obtained by taking $t$ copies of $H$ and identifying all the vertices that correspond to $u$, for each $u \in U$.
		\end{defn}

		\begin{defn} \label{defn:r}
			Given graphs $H$ and $T$, let $r(H, T)$ be the maximum number of components in $H \setminus U$, over subsets $U \subseteq V(H)$ for which the $(U, |T|)$-blow-up of $H$ is $T$-free. (See \Cref{fig:blow-up} for an example.)
		\end{defn}

		\begin{figure}[h] 
			\includegraphics[scale = 1]{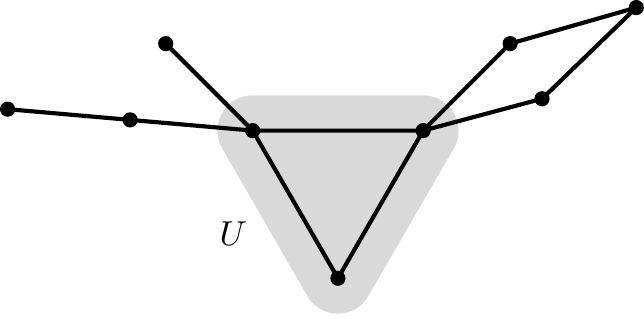}
			\hspace{1cm}
			\includegraphics[scale = 1]{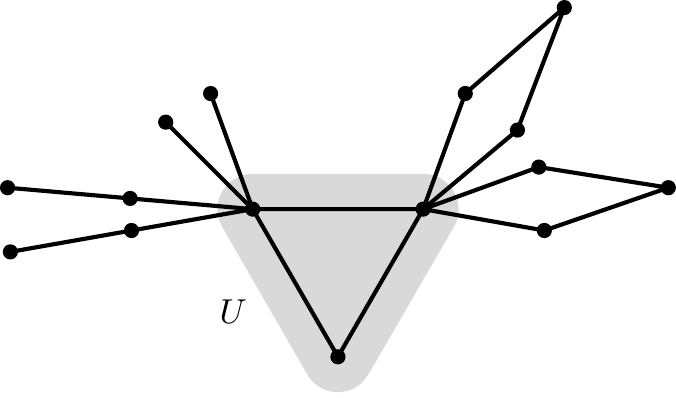}
			\caption{A graph $H$ and a subset $U \subseteq V(H)$ and the $(U, 2)$-blow-up of $H$.}
			\label{fig:blow-up}
		\end{figure}

		In the following theorem we estimate $\Ex(n, H, T)$, where $T$ is a tree, in terms of the value $r(H, T)$. Note that \Cref{thm:main} follows immediately.

		\begin{thm} \label{thm:main-expicit}
			Let $H$ be a graph and let $T$ be a tree. Then $\Ex(n, H, T) = \Theta(n^r)$, where $r = r(H, T)$.
		\end{thm}

		The lower bound follows quite easily from the definition of $r(H,T)$, so the main work goes into proving the matching upper bound. In \cite{alon-shikhelman} Alon and Shikhelman proved the same statement under the additional assumption that $H$ is a tree. In order to prove the upper bound, they showed that a graph $G$ which is $T$-free and has at least $c \cdot n^r$ copies of $H$ contains a $(U, |T|)$-blow-up of $H$, for some $U \subseteq V(H)$ such that $H \setminus U$ has at least $r+1$ components. Since $G$ is $T$-free, it follows that the $(U, |T|)$-blow-up is also $T$-free, which implies that $G$ has fewer than $c \cdot n^{r(H,T)}$ copies of $H$, as required. Our ideas are somewhat similar, but we do not prove that $G$ contains such a blow-up. Instead, we find a subgraph $G'$ of $G$ with many $H$-copies that behaves somewhat similarly to a $(U, |T|)$-blow-up of $H$, for some $U$ for which the number of components of $H \setminus U$ is larger than $r$. We then show that if the blow-up contains a copy of $T$ then so does $G'$. It again follows that the number of $H$-copies in $G$ is smaller than $c \cdot n^{r(H,T)}$.

		\begin{proof}[ of \Cref{thm:main-expicit}]

			Let $r = r(H, T)$, $h = |H|$, $t = |T|$ and $m = \Ex(n, H, T)$. Our aim is to show that $m = \Theta(n^r)$.
			
			We first show that $m = \Omega(n^r)$. Indeed, let $U \subseteq V(H)$ be such that $H \setminus U$ has $r$ components and the $(U, t)$-blow-up of $H$ is $T$-free. Let $G$ be the $(U, n/h)$-blow-up of $H$. Note that $G$ is $T$-free; indeed, otherwise, since any $T$-copy in $G$ uses vertices from at most $t$ copies of $H$, it would follow that the $(U, t)$-blow-up of $H$ is not $T$-free. Additionally, the number of $H$-copies in $G$ is at least $(n / h)^r$ since, for every component in $H \setminus U$, we can choose any of the $n / h$ copies of it in $G$, and together with $U$ this forms a copy of $H$.   

			The remainder of the proof will be devoted to proving the upper bound $m = O(n^r)$. Suppose to the contrary that $m \ge c \cdot n^r$, for a sufficiently large constant $c$. Let $G$ be a $T$-free graph on $n$ vertices with $m$ copies of $H$.

			We wish to replace $G$ with a subgraph $G'$ that has many $H$-copies and is somewhat similar to a $(U, t)$-blow-up of $H$ for an appropriate $U$. We shall obtain the required subgraph in three steps.

			First, we replace $G$ with an $r$-partite subgraph that still has many $H$-copies. 
			To achieve this goal, pick a label in $V(H)$ uniformly at random
			for each vertex in $G$. Denote by $X$ the number of $H$-copies in $G$ for which each vertex $u \in V(H)$ is mapped to a vertex in $G$ that received the label $u$. It is easy to see that the $\E(X) = m / h^h$. It follows that there exists a partition $\{V_u\}_{u \in V(H)}$ of the vertices of $G$ for which $X \ge m / h^h$. Fix such a partition, and denote by $\HH_0$ the family of $H$-copies for which every $u \in V(H)$ is mapped to $V_u$ (so $|\HH_0| \ge m / h^h$). Let $G_0$ be the subgraph of $G$ whose edges are the union of the edges in all $H$-copies in $\HH_0$.

			Next, since $G_0$ is $T$-free (as it is a subgraph of $G$), it is $d$-degenerate; fix an ordering $<$ of $V(G_0)$ such that every vertex $u$ has at most $t$ neighbours that appear after $u$ in $<$.
			Every $H$-copy inherits an ordering of $V(H)$ from $<$. Denote by $<_H$ the most popular such ordering and let $\HH_1$ be the subfamily of $H$-copies in $\HH_0$ that received the ordering $<_H$ (so $|\HH_1| \ge m / (h^h h!)$).

			We now turn to the final step towards obtaining the required subgraph of $G$. Ideally, we would have liked to find a graph $H_2$, which is the union of $\Omega(m)$ copies of $H$ in $\HH_1$, and which satisfies the following: for every $uw \in E(H)$, either all vertices in $V_u$ have small degree into $V_w$, or all vertices in $V_u$ have much larger degree into $V_w$. Such a property would allow us to show that if a suitable $(U, t)$-blow-up of $H$ contains a copy of $T$, then so does $H_2$. However, it is not clear if such a family of $H$-copies exists. Instead, we aim for a sequence of graphs 
			$H_2^{(1)} \supseteq \ldots \supseteq H_2^{(t)}$ (each of which is a union of a large collection of $H$-copies in $\HH_1$) such that for every $uw \in H$, either all vertices in $V_u$ have small degree into $V_w$ in the graph $H_2^{(1)}$, or all non-isolated vertices in $H_2^{(i+1)}$ have much larger degree into $V_w$ in the graph $H_2^{(i)}$. Such a sequence still allows us to find a copy of $T$ in $H_2^{(1)}$, under the assumption that a certain $(U,t)$-blow-up of $H$ contains a copy of $T$, using the fact that $T$ is a tree.
			In order to find the required sequence of graphs, pick constants $t \ll c_1 \ll \ldots \ll c_{e(H)} \ll c$, and follow Procedure~\ref{alg:H} below.
			\algsetup{indent = 1.5em}
			\floatname{algorithm}{Procedure}
			\newcommand{\algorithmicbreak}{\textbf{break}}
			\newcommand{\BREAK}{\STATE \algorithmicbreak}
			\begin{algorithm}[h]
				\setstretch{1.1}
				\caption{Modifying $\HH_1$}
				\label{alg:H}
				\begin{algorithmic}
					\STATE{$\HH_1^{(1)} = \HH_1$.}				
					\STATE{Set $E_1$ to be the set of ordered pairs $\{uw : uw \in E(H), u >_H w\}$ (so $|E_1| = e(H)$).}
					\STATE{Set $i = 1$, $j = 1$.}
					\WHILE{$i \le e(H), j < t$}
							\STATE{For every $e = uw \in E_i$, let $B_e$ be the set of vertices in $V_u$ whose degree into $V_w$, in the graph formed by the union of $H$-copies in $\HH_i^{(j)}$, is at most $c_i$.}
							\IF{at least half the $H$-copies in $\HH_i^{(j)}$ avoid $\bigcup_{e \in E_i} B_e$}
								\STATE{Set $\HH_i^{(j+1)}$ to be the family of $H$-copies in $\HH_i^{(j)}$ that avoid $\bigcup_{e \in E_i} B_e$.}
								\STATE{$j \leftarrow j+1$.}
							\ELSE
								\STATE{Let $e \in E_i$ be such that at least $\frac{1}{2|E_i|}$ of the $H$-copies in $\HH_i^{(j)}$ are incident with $B_e$.}
								\STATE{Set $\HH_{i+1}^{(1)}$ to be the family of $H$-copies in $\HH_i^{(j)}$ that are incident with $B_e$.}
								\STATE{Set $E_{i+1} = E_i \setminus \{e\}$.}
								\STATE{$i \leftarrow i+1$, $j \leftarrow 1$.}	
							\ENDIF
					\ENDWHILE
				\end{algorithmic}
			\end{algorithm}
			 
			Note that the procedure ends either with $i = e(H) + 1$ and $E_1 = \emptyset$, or with $i \le e(H)$, $j = t$ and $|E_i| = e(H) - (i-1)$. Let $l$ be the value of $i$ at the end of the procedure. In the next claim we show that the latter case holds, i.e.\ $l \le e(H)$. 

			\begin{claim} \label{claim:l-small}
				$l \le e(H)$.
			\end{claim}

			\begin{proof}
				Suppose that $l > e(H)$. It follows that for every $uw \in E(H)$, every vertex in $V_u$ sends at most $c_{l-1}$ edges into $V_w$.

				Let $a$ be the number of connected components in $H$. Note that the $(\emptyset, t)$-blow-up of $H$ is $T$-free (it is a disjoint union of copies of $H$, and we may assume that $H$ is $T$-free, as otherwise $m = 0$ and we are done immediately), and has $a$ components. Thus, by \Cref{defn:r}, we have $a \le r$.
								
				Write $\F = \HH_l^{(1)}$ and let $F$ be the graph formed by taking the union of all $H$-copies in $\F$. Note that $|\F| \ge (\frac{1}{2e(H)})^{t \cdot e(H)}|\HH_1| \ge (\frac{1}{2e(H)})^{t \cdot e(H)}\frac{1}{h^h h!} m > \frac{1}{\sqrt{c}} \cdot m$ (as $c$ is large).
				
				In order to upper-bound the number of $H$-copies in $\F$, let $U$ be a set of vertices in $H$ that contains exactly one vertex from each component. Trivially, there are at most $n^a$ ways to map each vertex $u \in U$ to a vertex in $V_u$. Fix such a mapping. Let $w$ be a vertex in $H$ with a neighbour $u \in U$. Since $w$ is mapped to one of the neighbours in $V_w$ of the vertex that $u$ is mapped to, there are at most $c_{l-1}$ vertices that $w$ can be mapped to. Similarly, if $w$ is in distance $i$ from a vertex $u \in U$, there are at most $(c_{l-1})^i$ vertices that $w$ can be mapped to. By choice of $U$, every vertex in $H$ is in distance at most $h$ from some vertex in $U$, hence there are at most $(c_{l-1})^{h^2}$ ways to complete the embedding of $U$ to an $H$-copy in $\F$. In total, we find that $|\F| \le (c_{l-1})^{h^2} \cdot n^a < \sqrt{c} \cdot n^a$. 
				
				Putting the two bounds on $|\F|$ together, we have $m < c n^a \le c n^r$, a contradiction. It follows that $l \le e(H)$, as desired. 
			\end{proof}

			From now on, we may assume that $l \le e(H)$, which means that $\HH_l^{(t)}$ has been defined. Write $\F_i = \HH_l^{(i)}$, and denote by $F_i$ the graph formed by the union of all $H$-copies in $\F_i$. 
			Let $D$ be the directed graph on vertex set $V(H)$ with edges $\{uw, wu : uw \in E(H)\}$ (so each edge in $H$ is replaced by two directed edges, one in each direction). We $2$-colour the edges of $D$: colour the edges in $E_l$ red and colour the remaining edges blue. (Note that if $uw$ is red then $wu$ is blue.) Denote the graph of blue edges by $D_B$ and the graph of red edges by $D_R$.
			We shall use the following properties of $F_i$ and $\F_i$.

			\begin{claim} \label{claim:properties-F-i}

				The following two properties hold for $2 \le i \le t$.
				\begin{enumerate}[label = {\rm (\roman*)}]
					\item
						every non-isolated vertex in $F_i$ is contained in an $H$-copy in $\F_{i-1}$,
					\item
						let $uw$ be a red edge in $D$ and let $S = \bigcup_{v:\text{ there is a blue path from $v$ to $w$}}V_v$. Then for every non-isolated vertex $x \in V_u$ there is a collection of $t$ copies of $H$ in $\F_{i-1}$ that contain $x$ and whose intersections with $S$ are pairwise vertex-disjoint.
				\end{enumerate}

			\end{claim}

			\begin{proof}

				The first property follows immediately from the definition of $F_i$ as the union of $H$-copies in $\F_i$: if a vertex is non-isolated in $F_i$ it is also non-isolated in $F_{i-1}$, and thus it must be contained in some $H$-copy in $\F_{i-1}$.
				
				Now let us see why the second property holds. Note that the directed edge $uw$ is in $E_l$ as $uw$ is a red edge in $D$. Thus, by definition of $\F_i$, the non-isolated vertex $x$ sends at least $c_l$ edges into $V_w$ in the graph $F_{i-1}$. This means that there is a collection of at least $c_l$ copies of $H$ in $\F_{i-1}$ that contain $x$, each of which uses a different edge from $x$ to $V_u$; denote this family of $H$-copies by $\F$. We claim that every $H$-copy in $\F$ intersects in $S$ with at most $h (c_{l-1})^h$ other $H$-copies in $\F$. Indeed, there are at most $h$ ways to choose an intersection point; suppose that the intersection is in $y \in V_v \subseteq S$. By choice of $S$, there is a path $(v_0 = v, v_1, \ldots, v_k = w)$ from $v$ to $w$ in $D_B$. This means that the degree (in $F_{i-1}$) of any vertex in $V_{v_j}$ into $V_{v_{j+1}}$ is at most $c_{l-1}$. Thus, there are at most $(c_{l-1})^k \le (c_{l-1})^h$ vertices in $V_w$ that can be in the same $H$-copy in $\F$ as $y$. Since each $H$-copy in $\F$ uses a different vertex of $V_w$, it follows that at most $(c_{l-1})^h$ copies of $H$ in $\F$ contain $y$, and in total there are at most $h (c_{l-1})^h$ copies of $H$ in $\F$ that intersect any single $H$-copy in $\F$. Since the total number of $H$-copies in $\F$ is $c_l \ge t (h (c_{l-1})^h+1)$, there is a collection of $t$ copies of $H$ in $\F$ whose intersections with $S$ are pairwise disjoint, as required. 
			\end{proof}

			We now wish to find a particular subset $U \subseteq V(H)$ such that the $(U, t)$-blow-up of $H$ behaves similarly to the sequence of graphs $F_1, \ldots, F_t$. The set $U$ will be defined in terms of a certain set $A \subseteq V(H)$, which we define now.
			Let $\P$ be a partition of $V(H)$ into strongly connected components according to $D_B$. Pick a set $A \subseteq V(H)$ that satisfies the following properties.
			\begin{enumerate}[label = \rm (\alph*)]
				\item \label{itm:reachable}
					every vertex in $D_B$ is reachable from $A$, i.e.\ for every vertex in $D_B$ there is a blue path to it from $A$,
				\item \label{itm:min-size}
					$|A|$ is minimal among sets that satisfy Condition \ref{itm:reachable},
				\item \label{itm:max-reachable}
					among sets that satisfy Conditions \ref{itm:reachable} and \ref{itm:min-size}, $A$ maximises 
					\begin{equation} \label{eqn:sum-reachable}
						\sum_{u \in A} (\#\text{ vertices reachable from $u$}).
					\end{equation}
			\end{enumerate}

			Let $W$ be the set of vertices in $V(H)$ that are in the same part of $\P$ as one of the vertices in $A$, and let $U = V(H) \setminus W$. In the following two claims we list some useful properties of $A$, $U$ and $W$.
			\begin{claim} \label{claim:properties-A}
				The following properties hold.
				\begin{enumerate}[label = \rm (\roman*)]
					\item \label{itm:one-from-part}
						$A$ contains at most one vertex from each part of $\P$,
					\item \label{itm:no-edges-between-parts}
						there are no edges of $D$ between distinct parts of $\P$ that are contained in $W$,
					\item \label{itm:no-blue-edges-back}
						there are no blue edges from $U$ to $W$.
				\end{enumerate}
			\end{claim}

			\begin{proof}
				Property \ref{itm:one-from-part} clearly holds because of the minimality of $|A|$ and the fact that for every part $X \in \P$, the set of vertices reachable from $X$ is the same as the set of vertices reachable from any individual vertex $x \in X$.

				For \ref{itm:no-edges-between-parts}, suppose that there is an edge $uw$ in $D$ with $u,w \in W$; without loss of generality $uw$ is blue. If we remove from $A$ the vertex from the same part of $\P$ as $w$, we obtain a smaller set that still satisfies Condition \ref{itm:reachable} above, a contradiction to the minimality of $A$.

				Now suppose that Property \ref{itm:no-blue-edges-back} does not holds, i.e.\ there is a blue edge $uw$ with $u \in U$ and $w \in W$. Let $A'$ be the set obtained from $A$ by removing the vertex $w'$ that is in the same part of $\P$ as $w$ and adding $u$. Note that every vertex that is reachable from $A$ is also reachable from $A'$. Moreover, every vertex that is reachable from $w'$ is also reachable from $u$, but $u$ is not reachable from $w'$, because otherwise $u$ and $w'$ are in the same strongly connected component, and hence in the same part of $\P$. It follows that  
				\begin{equation*}
					\sum_{u \in A'} (\#\text{ vertices reachable from $u$}) >
					\sum_{u \in A} (\#\text{ vertices reachable from $u$}),
				\end{equation*}
				a contradiction to the maximality property of $A$.
			\end{proof}

			\begin{claim} \label{claim:upper-bd-A}
				$|A| > r$.
			\end{claim}

			\begin{proof}
				Suppose that $|A| \le r$. As in the proof of \Cref{claim:l-small}, there are at most $n^{|A|}$ to embed $A$ in $V(F_1)$ (recall that $F_1$ is the union of all $H$-copies in $\F_1 = \HH_l^{(1)}$) in such a way that every $a \in A$ is sent to $V_a$. Fix such an embedding, and let $u \in V(H)$. Because there is a blue path from $A$ to $u$ (by Condition \ref{itm:reachable} in the definition of $A$), there are at most $(c_{l-1})^{h}$ vertices that $u$ could be mapped to which may form an $H$-copy in $\F_1$ together with the vertices that $A$ is mapped to. Thus, in total there are at most $(c_{l-1})^{h^2} \cdot n^r$ copies of $H$ in $\F_1$. As in the proof of \Cref{claim:l-small}, this implies that there are fewer than $c \cdot n^r$ copies of $H$ in $G$, a contradiction.
			\end{proof}

			Let $\Gamma$ be the $(U, t)$-blow-up of $H$ (see \Cref{defn:blow-up} and Figure~\ref{fig:blow-up}). Denote its vertices by $U \cup \left(\bigcup_{i \in [t]} W_i \right)$, where the $W_i$'s are copies of the set $W$ (so $\Gamma[U \cup W_i]$ induced a copy of $H$ for every $i \in [t]$). For every vertex $x$ in $\Gamma$, denote by $\phi(x)$ the vertex in $H$ that it corresponds to.
			By \Cref{claim:properties-A}~\ref{itm:one-from-part} and \ref{itm:no-edges-between-parts}, $H \setminus U$ consists of $|A| > r$ components. Because $r = r(H, T)$ (see \Cref{defn:r}), $\Gamma$ contains a copy of $T$. 
			
			Our final aim is to show that $G$ contains a copy of $T$, a contradiction to the assumptions on $G$. Consider a specific embedding of $T$ in $\Gamma$. Let $\{X_1, \ldots, X_k\}$ be a partition of $V(T)$, such that for every $i \in [k]$ the subgraph $T[X_i]$ is a maximal non-empty subtree of $T$ that is contained either in $W_j$, for some $j$, or in $U$. We assume, for convenience, that the ordering is such that there is an edge between $X_i$ and $X_1 \cup \ldots \cup X_{i-1}$ for every $i \in [k]$; in fact, there would be exactly one such edge as $T$ is a tree. By choice of the $X_i$'s and by definition of $\Gamma$, this edge must be an edge between some set $W_j$ and $U$. In order to reach the required contradiction, we prove the following game.
		
			\begin{claim}
				For every $i \in [k]$ there is a copy of $T[X_1 \cup \ldots \cup X_i]$ in $F_{t-(i-1)}$ such that $x$ is mapped to $V_{\phi(x)}$ for every $x \in X_1 \cup \ldots \cup X_i$.
			\end{claim}

			\begin{proof}

				We prove the statement by induction on $i$.
				For $i = 1$, the statement can easily be seen to hold, by picking any $H$-copy in $\F_t$, and mapping each vertex of $X_1$ to the corresponding vertex in the copy of $H$.

				Now suppose that the statement holds for $i$; let $f_i : X_1 \cup \ldots \cup X_i \rightarrow V(F_{t-(i-1)})$ be the corresponding mapping of the vertices. Now, there are two possibilities to consider: $X_{i+1} \subseteq U$ or $X_{i+1} \subseteq W_j$ for some $j$. 
				
				Let us consider the first possibility. Let $uw$ be the edge between $X_1 \cup \ldots \cup X_{i}$ and $X_{i+1}$, where $u \in U$ and $w \in W_j$ for some $j$ (so $u \in X_{i+1}$ and $w \in X_1 \cup \ldots \cup X_{i}$). We may assume that $f_i(w)$ is non-isolated. Indeed, if $|X_1 \cup \ldots \cup X_i| \ge 2$, this is clear since $T[X_1 \cup \ldots \cup X_i]$ spans a tree. Otherwise, we must have that $i = 1$ and $|X_1| = 1$, but then we can choose $f_1(w)$ to be a non-isolated vertex in $V_w$. As $f_i(w)$ is non-isolated, by \Cref{claim:properties-F-i} (and the fact that $i \le k \le t$) there is an $H$-copy in $\F_{t-i}$ that contains $f_i(w)$; denote the corresponding embedding by $g : V(H) \rightarrow V(F_{t-i})$. We define $f_{i+1} : X_1 \cup \ldots \cup X_{i+1} \rightarrow V(F_{t-i})$ simply by
				\begin{align*}
					f_{i+1}(x) = \left\{
						\begin{array}{ll}
							f_i(x) & x \in X_1 \cup \ldots \cup X_i \\
							g(x) & x \in X_{i+1}.
						\end{array}
					\right.
				\end{align*}
				In order to show that $f_{i+1}$ is an embedding with the required properties, we need to show that it has the following three properties: it maps edges in $T[X_1 \cup \ldots \cup X_{i+1}]$ to edges in $F_{t-i}$; $f_{i+1}(x) \in V_{\phi(x)}$ for every $x \in X_1 \cup \ldots \cup X_{i+1}$; and $f_{i+1}$ is injective.

				We first show that $f_{i+1}$ preserves edges. This follows because $f_i$ and $g$ preserve edges (this holds for $g$ by definition, and holds for $f_i$ because it sends edges to edges of $F_{t-(i-1)}$ which is a subgraph of $F_{t-i}$) so edges inside $X_1 \cup \ldots \cup X_i$ and inside $X_{i+1}$ are mapped to edges in $F_{t-i}$, and moreover by choice of $g$ the only edge between these two sets is mapped to an edge of $F_{t-i}$.

				Next, we note that for every $x \in X_1 \cup \ldots \cup X_{i+1}$, we have $f_{i+1}(x) \in V_{\phi(x)}$. This is because this holds for both $f_i$ (by assumption) and $g$ (as $g$ corresponds to an $H$-copy in $\F_{t-i}$).

				Finally, we show that $f_{i+1}$ is injective. As both $f_i$ and $g$ are injective, it suffices to show that $g(x) \neq f_i(y)$ for every $x \in X_{i+1}$ and $y \in X_1 \cup \ldots \cup X_i$. This holds because $\phi(x) \neq \phi(y)$ (since $x$ is in $U$, it is the only vertex in $X_1 \cup \ldots \cup X_{i+1}$ with $\phi(x) = x$) and because $x$ and $y$ are mapped to $V_{\phi(x)}$ and $V_{\phi(y)}$, respectively, and these two sets are disjoint.
				
				Now we consider the second possibility, namely that $X_{i+1} \subseteq W_j$ for some $j$. Let $uw$ be the edge between $X_1 \cup \ldots \cup X_i$ and $X_{i+1}$, where $u \in U$ and $w \in W_j$ (so $w \in X_{i+1}$). By \Cref{claim:properties-A}~\ref{itm:no-blue-edges-back}, the edge $uw$ is red. Hence, by \Cref{claim:properties-F-i}, there is a collection of $t$ copies of $H$ in $\F_{t-i}$ that contain $f_i(u)$ and whose intersections with $S = \bigcup_{v: \text{ there is a blue path from $v$ to $w$}} V_v$ are pairwise vertex-disjoint. As $|X_1 \cup \ldots \cup X_i| < t$, it follows that there is an $H$-copy in $\F_{t-i}$ that contains $f_i(w)$ and whose intersection with $S$ is disjoint of $f_i(X_1 \cup \ldots \cup X_i)$; denote the corresponding embedding of $H$ by $g : V(H) \rightarrow V(F_{t-i})$. As before, define $f_{i+1} : X_1 \cup \ldots \cup X_{i+1} \rightarrow V(F_{t-i})$ by
				\begin{align*}
					f_{i+1}(x) = \left\{
						\begin{array}{ll}
							f_i(x) & x \in X_1 \cup \ldots \cup X_i \\
							g(x) & x \in X_{i+1}.
						\end{array}
					\right.
				\end{align*}
				As before, $f_{i+1}$ maps edges of $T[X_1 \cup \ldots X_{i+1}]$ to edges of $F_{t-i}$, and it sends every $x \in X_1 \cup \ldots \cup X_{i+1}$ to $V_{\phi(x)}$. Moreover, by choice of $g$ and since $g(X_{i+1}) \subseteq S$, we find that $g(X_{i+1})$ and $f_i(X_1 \cup \ldots \cup X_i)$ are disjoint. Since $f_i$ and $g$ are both injective, it follows that $f_{i+1}$ is injective. This completes the proof of the induction, and thus of the claim.
			\end{proof}

			By taking $i = k$ in the previous claim, we find that $F_{t-(k-1)}$ contains a copy of $T$. But $F_{t-(k-1)} \subseteq F_1 \subseteq G$ (note that $k \le t$), so $G$ has a copy of $T$, a contradiction. It follows that the number of $H$-copies in $G$ is at most $c \cdot n^{r(H,T)}$, as required.
	\end{proof}

	\section{Conclusion} \label{sec:conc}

		In this paper we showed how to determine, up to a constant factor, the function $\Ex(n, H, T)$ whenever $T$ is a tree. It would, of course, be interesting to determine this function completely, or at least asymptotically. While this may be hopeless in general, in some special cases this task may not be out of reach. For example, Alon and Shikhelman \cite{alon-shikhelman} consider the special case where $H = K_h$ for some $h < t$ and $t = |T|$. They ask if the $n$-vertex graph, which is the union of $\floor{n/t}$ disjoint cliques of size $t$, and perhaps one smaller clique on the remainder maximises the number of copies of $K_h$ among all $T$-free graphs on $n$ vertices. This question generalises a question of Gan, Loh and Sudakov \cite{gan-loh-sudakov}, who considered the case where $T$ is a star on $t$ vertices. In other words, they were interested in maximising the number of cliques of size $h$ among $n$-vertex graphs with maximum degree smaller than $t$. They proved that the aforementioned construction of disjoint cliques is the unique extremal example when $n \le 2t$, thus proving a conjecture of Engbers and Galvin \cite{engbers-galvin}. The question whether this construction is best for larger values of $n$ remains open.

		For other questions regarding the value of $\Ex(n, H, F)$, where $F$ need not be a tree, see \cite{alon-shikhelman}.

\bibliography{t-free}
\bibliographystyle{amsplain}
\end{document}